\documentclass{article}

\usepackage[all]{xy}
 \usepackage[usenames,dvipsnames]{pstricks}
 \usepackage{epsfig}
 \usepackage{pst-grad} 
 \usepackage{pst-plot} 
\usepackage{ebezier}
\usepackage{mathrsfs}

\usepackage[
        colorlinks, citecolor=red,
        backref,
        pdfauthor={MDJ,JK},
        pdftitle={S1}
]{hyperref}

\usepackage{latexsym}\usepackage{amsfonts}\usepackage{amssymb}\usepackage{amsthm}
\usepackage{amsmath}\usepackage{newlfont}\usepackage{graphicx}\usepackage{doc}
\usepackage{color}\usepackage{multicol}\usepackage{epic}\usepackage{eepic}
\usepackage{ebezier}
\newtheorem{thrm}{Theorem}
\newtheorem{lem}[thrm]{Lemma}
\newtheorem{cor}[thrm]{Corollary}

\newtheorem{prop}[thrm]{Proposition}
\newtheorem{exam}[thrm]{Example}

\newtheorem*{theorem-non}{Theorem}

\newtheorem*{namedthm}{\namedthmname}
\newcounter{namedthm}



\include{epsf}

\def \Dj{\mbox{\raise0.3ex\hbox{-}\kern-0.4em D}}

\begin{document}

\title{A flexibility result for polynomial entropy of pointwise periodic homeomorphisms}

\author{Ma\v{s}a \Dj ori\'c\thanks{Corresponding author: Ma\v{s}a \Dj ori\'c.}\\
Matemati\v{c}ki institut SANU\\
Knez Mihailova 36\\
11000 Beograd\\
Serbia\\
masha@mi.sanu.ac.rs\\ \and
Jelena Kati\'c\\
Matemati\v cki fakultet\\
Studentski trg 16\\
11000 Beograd\\
Serbia\\
jelena.katic@matf.bg.ac.rs \and
Milan Peri\'c\\
Matemati\v cki fakultet\\
Studentski trg 16\\
11000 Beograd\\
Serbia\\
milan.peric@matf.bg.ac.rs
}

\maketitle

\begin{abstract} 

We construct a family of continua and pointwise periodic homeomorphisms realizing arbitrary polynomial entropy values in $[0,+\infty]$. In particular, this provides examples of pointwise periodic homeomorphisms with positive polynomial entropy. This contrasts with the fact that pointwise periodic homeomorphisms on connected manifolds and local dendrites have zero polynomial entropy.

\end{abstract}

\medskip


\section{Introduction}

A useful invariant for studying dynamical systems with zero topological entropy is polynomial entropy. While both topological and polynomial entropy measure orbit complexity, they do so on different scales. Topological entropy describes the asymptotic exponential growth of orbit complexity, whereas polynomial entropy detects growth occurring at a polynomial rate. Consequently, polynomial entropy provides a finer classification within the class of systems whose topological entropy vanishes. To illustrate this distinction, compare an irrational rotation of the circle with a homeomorphism that possesses both periodic and wandering points. Although each of these systems has zero topological entropy, the rotation is intuitively much less complex. A result of Labrousse shows that polynomial entropy is capable of distinguishing between such examples (see Theorem 1 in~\cite{L}) and that the obtained numerical values align with our intuition regarding complexity.

Topological and polynomial entropy share several fundamental features. Both are invariant under topological conjugacy, depend only on the underlying topology rather than on the particular compatible metrics, satisfy the finite union property, and they both admit a product formula. Nevertheless, important differences remain. Certain classical properties of topological entropy, including the power formula, the $\sigma$-union property, and the variational principle, need not hold for polynomial entropy (see~\cite{M1,M2}). In addition, topological entropy is completely determined by the restriction of the system to its non-wandering set, which is a closed invariant subset. Polynomial entropy behaves differently: it can reflect dynamical behaviour occurring outside the non-wandering set. Thus, whereas topological entropy is insensitive to the wandering dynamics, polynomial entropy detects and quantifies the contribution of the wandering part to the complexity of the system. In this paper, however, we deal exclusively with systems that have no wandering points.

It is clear that periodic homeomorphisms have both topological and polynomial entropy equal to zero. Interestingly, there exist pointwise periodic homeomorphisms, for which $\mathrm{Per}(f)=X$, but which are not periodic. For these homeomorphisms it is rather easy to deduce that the topological entropy is equal to zero (due to the $\sigma$-union property), but it is not clear if, and under which additional assumptions, polynomial entropy is equal to zero as well. The first two authors in \cite{DK} prove that homeomorphisms without wandering points on local dendrites have zero polynomial entropy. Also, a well known theorem by Montgomery in \cite{M} states that the pointwise periodic homeomorphisms on connected manifolds have to be periodic, and therefore have zero polynomial (and topological) entropy.

We construct an example of a pointwise periodic homeomorphism on a continuum, inspired by a construction in \cite{AOM}, which has positive polynomial entropy. Furthermore, we obtain a flexibility result - for every $a\in[0,+\infty]$ there exists a pointwise periodic homeomorphism $f:X\to X$ on a continuum $X$, such that $h_{\mathrm{pol}}(f)=a$. Let us formulate our main result in the following theorem:

\begin{theorem-non}
For every $a\in[0,+\infty]$ there exists a compact connected metric space $X$ and a pointwise periodic homeomorphism $f:X\to X$ such that $h_{\mathrm{pol}}(f)=a$.
\end{theorem-non}

\section{Polynomial entropy}\label{subsec:entropy}

Suppose that $(X,d)$ is a compact metric space, and $f:X\rightarrow X$ is continuous. Denote by $d_n^f(x,y)$ the dynamic metric (induced by $f$ and $d$):
$$
d_n^f(x,y)=\max\limits_{0\leqslant k\leqslant n-1}d(f^k(x),f^k(y)).
$$
For $\varepsilon>0$, we say that a finite set $E\subset X$ is \textit{$(n,\varepsilon)$-separated} if for every $x,y \in E$ it holds $d_n^f(x,y)\geq \varepsilon$. Let $\mathrm{sep}(n,\varepsilon)$ denote the maximal cardinality of an $(n,\varepsilon)$-separated set $E$.

\begin{def}\label{def:pol_ent} The \textit{polynomial entropy} of the map $f$ on the compact metric space $X$ is defined by
$$
h_{\mathrm{pol}}(f,X)=\lim\limits_{\varepsilon \rightarrow 0}\limsup\limits_{n\rightarrow \infty}\frac{\log \mathrm{sep}(n,\varepsilon)}{\log n}.
$$
\end{def}

We can also define the polynomial entropy as follows. Let $\mathrm{span}(n,\varepsilon)$ denote the minimal number of balls of radius $\varepsilon$ (with respect to $d_n^f$) that cover $X$. Denote by $\mathrm{cov}(n,\varepsilon)$ the minimal number of sets $X_j$ such that the diameters (with respect to $d_n^f$) of $X_j$ are smaller than $\varepsilon$ and $X\subseteq\cup_{j=1}^mX_j$.

From the following sequence of inequalities
$$\mathrm{cov}(n,2\varepsilon)\leqslant \mathrm{span}(n,\varepsilon)\leqslant \mathrm{sep}(n,\varepsilon)\leqslant \mathrm{cov}(n,\varepsilon)$$ we conclude that
$$h_{\mathrm{pol}}(f,X)=\lim\limits_{\varepsilon \rightarrow 0}\limsup\limits_{n\rightarrow \infty}\frac{\log \mathrm{span}(n,\varepsilon)}{\log n}=\lim\limits_{\varepsilon \rightarrow 0}\limsup\limits_{n\rightarrow \infty}\frac{\log \mathrm{cov}(n,\varepsilon)}{\log n}.$$

We often abbreviate $h_{\mathrm{pol}}(f):=h_{\mathrm{pol}}(f,X)$. We list some properties of the polynomial entropy that are important for our computations (for proofs see Propositions $1-4$ in \cite{M2}):
\begin{itemize}
\item[(1)] $h_{\mathrm{pol}}(f^k)=h_{\mathrm{pol}}(f)$, for any $k\geqslant 1$.
\item[(2)] If $X=\bigcup_{j=1}^mX_j$ where $X_j$ are closed and $f$-invariant, then $h_{\mathrm{pol}}(f,X)=\max\{h_{\mathrm{pol}}(f,X_j)\mid j=1,\ldots,m\}$\label{finite-union}.\label{(3)}
\item[(3)] If $f:X\to X$, $g:Y\to Y$ and $f\times g\;:X\times Y\to X\times Y$ is defined as $f\times g (x,y):=(f(x),g(y))$, then $h_{\mathrm{pol}}(f\times g)=h_{\mathrm{pol}}(f)+h_{\mathrm{pol}}(g)$.
\item[(4)] $h_{\mathrm{pol}}(f)$ does not depend on a metric but only on the induced topology.
\item[(5)] $h_{\mathrm{pol}}(\cdot)$ is a \textit{conjugacy invariant} (meaning if $f:X\to X$, $g:X'\to X'$, $\varphi:X\to X'$ is a homeomorphism of compact spaces and $g\circ\varphi=\varphi\circ f$, then $h_{\mathrm{pol}}(f)=h_{\mathrm{pol}}(g)$).
\item[(6)] If $f:X\to X$ and $g:X'\to X'$ are \textit{semi-conjugated}, meaning that $\varphi:X\to X'$ is a continuous surjective map of compact spaces and $g\circ\varphi=\varphi\circ f$, then $h_{\mathrm{pol}}(f)\geqslant h_{\mathrm{pol}}(g)$.
\end{itemize}

A set $A\subset X$ is \textit{wandering} if $f^n(A)\cap A=\emptyset$, for all integer $n\geqslant 1$. A point $p\in X$ is \textit{wandering} if there exists a wandering neighbourhood $U\ni p$. \label{wan-pt}A point that is not wandering is said to be \textit{non-wandering}. 
We denote the set of all periodic points by $\mathrm{Per}(f)$.

\section{Pointwise periodic homeomorphisms}

Let $X$ be a compact metric space. We say that a homeomorphism $f:X\to X$ is \textit{pointwise periodic} if for every $x\in X$ there is $n(x)\in\mathbb{N}$ such that $f^{n(x)}(x)=x$. In this case, $\mathrm{Per}(f)=X$, so there aren't any wandering points. Since we can write $X=\bigcup\limits_{n\in\mathbb{N}}P_n$, where $P_n$ is the set of points in $X$ with minimal period equal to $n$, using the $\sigma-$union property of topological entropy and the fact that $\mathrm{h_{top}}(f,P_n)=0$, we deduce that $\mathrm{h_{top}}(f,X)=0$.

 We say that a homeomorphism $f:X\to X$ is \textit{periodic} if there exists $n\in\mathbb{N}$ such that $f^{n}(x)=x$, for all $x\in\mathbb{X}$. There are many examples of pointwise periodic homeomorphisms that are not periodic. However, when $X$ is a connected manifold this is not the case (see \cite{M}):
\begin{thrm}
Every pointwise periodic homeomorphism on a connected manifold is periodic.
\end{thrm}

One of the most natural examples of non-periodic pointwise periodic homeomorphisms is on a dendrite (see Example \ref{ex:dendrite}). We say that a space $D$ is a {\it dendrite} if $D$ is a locally connected continuum (a nonempty connected compact metric space) containing no simple closed curves. Some basic properties of dendrites and dendrite maps can be found, for example, in \cite{JB}.

A \textit{local dendrite} is a continuum such that its every point has a dendrite neighbourhood. Every local dendrite is a \textit{Peano continuum} (locally connected continuum) which has a finite number of circles. Dendrites and graphs are local dendrites. It is known that Peano continua are arcwise connected and locally arcwise connected (see \cite{SN}). The proof of the following theorem in a more general setting for regular curves can be found in \cite{GS}.

\begin{prop}
If $f:X\to X$ is a homeomorphism on a local dendrite $X$, then the topological entropy of $f$ is equal to zero.
\end{prop}

\begin{exam}\label{ex:dendrite}
Let us begin by saying that an $n-$od space is a union of $n$ straight line segments, any two of which intersect only at their common endpoint $x_0$.

Let $X$ be a dendrite defined in the following way: $X$ is a union of countably many straight line segments in the plane, any two of which intersect only at their common endpoint $x_0$, and such that for every $\varepsilon>0$ there are finitely many segments whose length is greater than $\varepsilon$. $X$ can be represented as a union of countably many $a_n$-od spaces, where $(a_n)_{n\in\mathbb{N}}$ is some strictly increasing sequence.

We define $f:X\to X$ to be rotation $R_{a_n}$ on each $a_n$-od space. In this way, we obtained a pointwise periodic homeomorphism which is not periodic.
\end{exam}
\begin{figure}[h]
    \centering 
    \includegraphics[width=0.4\textwidth]{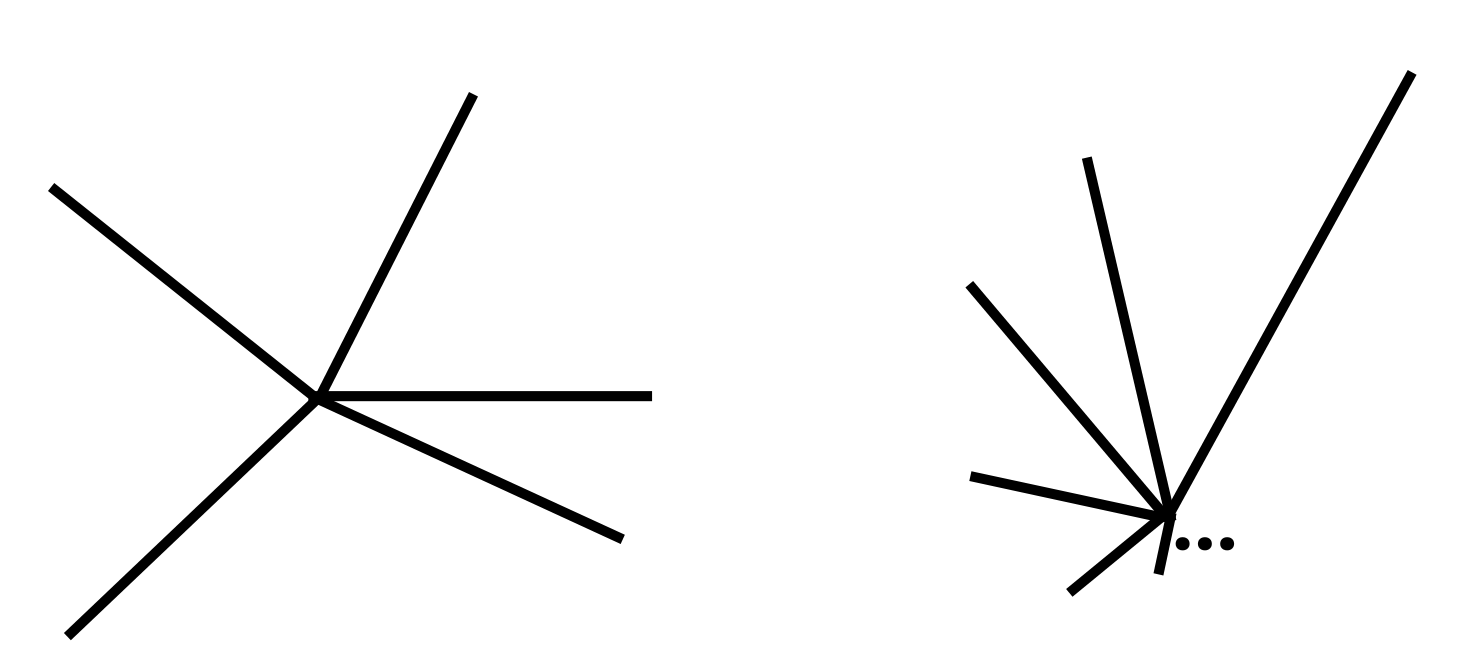}
    \caption{5-od on the left and $X$ on the right}
\end{figure}

Despite the fact that the map from the previous example is not periodic, its polynomial entropy is still equal to zero, as is shown by the following proposition. 

\begin{prop}
Let $X$ be a local dendrite and $f:X\to X$ a pointwise periodic homeomorphism. Then $h_{\mathrm{pol}}(f)=0$.
\end{prop}

\noindent{\it Proof.} 
For pointwise periodic homeomorphisms we have that $\mathrm{Per}(f)=X$, so $f$ possesses no wandering points. Using Theorem 4.2 in \cite{DK}, we conclude that $h_{\mathrm{pol}}(f)=0$.
 \qed

So far we have seen that for connected manifolds and local dendrites pointwise periodicity implies zero polynomial entropy. In the next section we will construct a continuum and a pointwise periodic homeomorphism which has positive polynomial entropy, showing that the following theorem is true:

\begin{thrm}\label{main}
For every $\alpha>2$ there exists a continuum $X$ and a pointwise periodic homeomorphism $f:X\to X$ such that $h_{\mathrm{pol}}(f)=\frac{1}{\alpha}$.
\end{thrm}

\begin{cor}\label{cor}
For every $a\in[0,+\infty)$ there exists a pointwise periodic homeomorphism $f:X\to X$ on a continuum $X$ such that $h_{\mathrm{pol}}(f)=a$.
\end{cor}

Before turning our attention to the proof of the main result, let us mention an important class of dynamical systems with zero topological entropy called {\it distal} dynamical systems. A homeomorphism $f:X\to X$ is called distal if for every two distinct points $x,y\in X$ there exists $\delta>0$ such that $d(f^n(x),f^{n}(y))>\delta$, for all $n\in\mathbb{Z}$. It is easy to see that pointwise periodic homeomorphisms are distal, since the orbit of every point is finite. Also, {\it equicontinuous} homeomorphisms are distal.

Recall that a homeomorphism $f:X\to X$ is equicontinuous if for all $\varepsilon>0$ there is a $\delta>0$ such that if $x,y\in X$, $d(x,y)<\delta$, then $d(f^n(x),f^n(y))<\varepsilon$, for all $n\in\mathbb{Z}$. For an equicontinuous homeomorphism on a compact metric space there is a compatible metric that makes it an isometry - $d'(x,y)=\sup_{n\in\mathbb{Z}}d(f^n(x),f^n(y))$. Therefore, we can conclude that the following proposition holds.
\begin{prop}\label{equi}
Let $f:X\to X$ be an equicontinuous homeomorphism on a compact metric $X$. Then $h_{\mathrm{pol}}(f)=0$.
\end{prop}

It is known that distal homeomorphisms do not necessarily have zero polynomial entropy, as is shown by an example in \cite{AOM}. However, the constructed example is a distal homeomorphism on a space which is not a manifold, nor is it connected. We may ask ourselves is it possibly true that distal homeomorphisms on connected manifolds have zero polynomial entropy. The answer is negative, as we see from the following proposition:

\begin{prop}\label{prop:torus}
Let $f:\mathbb{T}^2\to\mathbb{T}^2$, $f(x,y)=(x+y\,\mathrm{mod}\,1,y)$. Then $f$ is distal, and $h_{\mathrm{pol}}(f)=1$.
\end{prop}

The proof follows immediately from the following easy consequence of Proposition 2.6 in~\cite{M1}.

\begin{prop}
Let $f:\mathbb{T}^n\times \mathbb{T}^n\to\mathbb{T}^n\times \mathbb{T}^n$, $f(x,y)=(x+g(y),y)$, where $g:\mathbb{T}^n\to\mathbb{T}^n$ is $C^1$. Then 
$$h_{\mathrm{pol}}(f)=\max_{y\in\mathbb{T}^n}\mathrm{rang}\,g(y).$$
\end{prop}

\noindent {\it Proof.} Note that for any set $K\subset\mathbb{T}^n$, the set $\mathbb{T}^n\times K$ is $f$-invariant. Let us cover $\mathbb{T}^n$ with $k$ sets $B_j$, diffeomorphic to the unit $n$-dimensional closed ball $\mathbb{B}$. Then
$$h_{\mathrm{pol}}(f)=\max\{h_{\mathrm{pol}}(f|_{B_j})\mid j\in\{1,\ldots,k\}\}.$$
Now by applying Proposition 2.6 in~\cite{M1}, we have
$$h_{\mathrm{pol}}(f)=\max\{\max_{y\in B_j}\mathrm{rang}\,g(y)\mid j\in\{1,\ldots,k\}\}=\max_{y\in\mathbb{T}^n}\mathrm{rang}\,g(y).$$

\qed

\noindent{\it Proof of Proposition \ref{prop:torus}.} Since $g(y)=y$ is the identity map, it is clear that $h_{\mathrm{pol}}(f)=1$. Let us show that $f$ is indeed distal. 

Take two distinct points $p=(x_1,y_1)$ and $q=(x_2,y_2)$. If $y_1\neq y_2$, then $$\inf\limits_{n\in\mathbb{N}}(f^n(p),f^n(q))\geqslant d_{\mathbb{S}^1}(y_1,y_2)>0.$$ 

If $y_1=y_2=y$, then $$\inf\limits_{n\in\mathbb{N}} d(f^n(p),f^n(q))=\inf\limits_{n\in\mathbb{N}} d((x_1+ny\,\mathrm{mod}\,1,y),(x_2+ny\,\mathrm{mod}\,1,y))=d_{\mathbb{S}^1}(x_1,x_2)>0.$$
\qed

Similarly to the discussion about distal dynamical systems, we can wonder if zero polynomial entropy on a continuum implies that the system is equicontinuous. This is not the case, as we will see after introducing the techniques needed to prove the result (see Proposition \ref{prop:ekvi}).

\section{Main results}\label{subsec:main}

We will begin by treating a slightly more general problem, and then constructing the example which proves the main Theorem \ref{main} as a corollary. We will make a connection between the polynomial entropy and the ball-dimension, for a specific family of spaces. In order to facilitate the computation of $\mathrm{sep}(n,\varepsilon)$ and $\mathrm{span}(n,\varepsilon)$, we introduce a special metric (equivalent to the Euclidean metric) defined by polar coordinates. This is a modification of the {\it polar taxicab distance} defined in~\cite{KPKK} (see also references therein).

\subsection{Metric}

Let $\mathbb{S}^1=\mathbb{R}/\mathbb{Z}$. Define $d_{\mathbb{S}^1}$ in the standard way. More precisely,
let $\varphi,\psi\in\mathbb{S}^1$, and let $t,s\in[0,1)$ be the representatives of the corresponding equivalence classes. Define
$$d_{\mathbb{S}^1}(\varphi,\psi):=\min\{|t-s|,1-|t-s|\}.$$ 

We now define a metric on the disk $D$. Let $(\rho,\varphi)$ denote polar coordinates on $D$ and let $x=(\rho_1\cos(2\pi\varphi_1),\rho_1\sin(2\pi\varphi_1))$, $y=(\rho_2\cos(2\pi\varphi_2),\rho_2\sin(2\pi\varphi_2))$. Define
$$d_D(x,y):=|\rho_1-\rho_2|+\min\{\rho_1,\rho_2\}d_{\mathbb{S}^1}(\varphi_1,\varphi_2).$$ Although polar coordinates are not globally defined, the metric $d_D$ is well defined. Indeed, the point $(0,0)$ (in Euclidean coordinates) corresponds to $(0,\varphi)$ in polar coordinates for every $\varphi\in\mathbb{S}^1$. If $x=(0,\varphi)$ and $y=(\rho_2,\psi)$ (in polar coordinates), then:
$$d_D(x,y)=\rho_2$$ which does not depend on $\varphi$.

\begin{prop} $d_D$ is a metric on $D$.
\end{prop}

\noindent{\it Proof.} It is obvious that $d_D$ is symmetric and non-degenerate.

Let us prove the triangle inequality. For $x_j=(\rho_j,\varphi_j)$ want to show
$$d(x_1,x_2)\leqslant d(x_1,x_3)+d(x_3,x_2)$$ or, equivalently
\begin{equation}\label{eq:triangle}
\begin{aligned}
&|\rho_1-\rho_2|+\min\{\rho_1,\rho_2\}d_{\mathbb{S}^1}(\varphi_1,\varphi_2)\leqslant\\
&|\rho_1-\rho_3|+\min\{\rho_1,\rho_3\}d_{\mathbb{S}^1}(\varphi_1,\varphi_3)+|\rho_3-\rho_2|+\min\{\rho_3,\rho_2\}d_{\mathbb{S}^1}(\varphi_3,\varphi_2).\end{aligned}
\end{equation}

We will derive the proof by examining the order of the points $\rho_1,\rho_2$ and $\rho_3$.
Since the equation~(\ref{eq:triangle}) is symmetric in $\rho_1$ and $\rho_2$, we only consider the position of $\rho_3$ with respect to $\rho_1$ and $\rho_2$. More precisely, we distinguish between three cases:
\begin{itemize}
\item[(a)] $\rho_1\leqslant\rho_3\leqslant\rho_2$
\item[(b)] $\rho_1\leqslant\rho_2\leqslant\rho_3$
\item[(c)] $\rho_3\leqslant\rho_1\leqslant\rho_2$.
\end{itemize}
Suppose the order (a) holds. Then we have
$$\begin{aligned}
&d_D(x_1,x_2)=|\rho_1-\rho_2|+\rho_1d_{\mathbb{S}^1}(\varphi_1,\varphi_2)\\
&\stackrel{(\heartsuit)}\leqslant |\rho_1-\rho_3|+|\rho_3-\rho_2|+\rho_1\big[d_{\mathbb{S}^1}(\varphi_1,\varphi_3)+d_{\mathbb{S}^1}(\varphi_2,\varphi_3)\big]\\
&\stackrel{(\clubsuit)}{\leqslant}|\rho_1-\rho_3|+|\rho_3-\rho_2|+\rho_1d_{\mathbb{S}^1}(\varphi_1,\varphi_3)+\rho_3d_{\mathbb{S}^1}(\varphi_2,\varphi_3)\\
 &\stackrel{(\diamondsuit)}{=}|\rho_1-\rho_3|+|\rho_3-\rho_2|+\min\{\rho_1,\rho_3\}d_{\mathbb{S}^1}(\varphi_1,\varphi_3)+\min\{\rho_2,\rho_3\}
 d_{\mathbb{S}^1}(\varphi_2,\varphi_3)\\ &=
d_D(x_1,x_3)+d_D(x_3,x_2).\end{aligned}$$ The inequality $(\heartsuit)$ is the triangle inequalities for $|\cdot|$ in $\mathbb{R}$ and $d_{\mathbb{S}^1}$, while $(\clubsuit)$ and $(\diamondsuit)$ hold because of the presumed order $\rho_1\leqslant\rho_3\leqslant\rho_2$.

The case (b) is treated similarly:
$$\begin{aligned}
&d_D(x_1,x_2)=|\rho_1-\rho_2|+\rho_1d_{\mathbb{S}^1}(\varphi_1,\varphi_2)\\
&\leqslant |\rho_1-\rho_3|+|\rho_3-\rho_2|+\rho_1\big[d_{\mathbb{S}^1}(\varphi_1,\varphi_3)+d_{\mathbb{S}^1}(\varphi_2,\varphi_3)\big]\\
&\leqslant|\rho_1-\rho_3|+|\rho_3-\rho_2|+\rho_1d_{\mathbb{S}^1}(\varphi_1,\varphi_3)+\rho_2d_{\mathbb{S}^1}(\varphi_2,\varphi_3)\\
 &{=}|\rho_1-\rho_3|+|\rho_3-\rho_2|+\min\{\rho_1,\rho_3\}d_{\mathbb{S}^1}(\varphi_1,\varphi_3)+\min\{\rho_2,\rho_3\}
 d_{\mathbb{S}^1}(\varphi_2,\varphi_3)\\ &=
d_D(x_1,x_3)+d_D(x_3,x_2).\end{aligned}$$

Let us prove the case (c). Suppose $\rho_3\leqslant\rho_1\leqslant\rho_2$, so the inequality~(\ref{eq:triangle}) becomes:
$$
\rho_2-\rho_1+\rho_1d_{\mathbb{S}^1}(\varphi_1,\varphi_2)\leqslant
\rho_1-\rho_3+\rho_3d_{\mathbb{S}^1}(\varphi_1,\varphi_3)+\rho_2-\rho_3+\rho_3d_{\mathbb{S}^1}(\varphi_3,\varphi_2),$$
or, equivalently:
$$\rho_3\left[2-d_{\mathbb{S}^1}(\varphi_1,\varphi_3)-d_{\mathbb{S}^1}(\varphi_3,\varphi_2)\right]\leqslant\rho_1\left[2-d_{\mathbb{S}^1}(\varphi_1,\varphi_2)\right].$$ The last inequality obviously holds since:
$$\rho_3\leqslant\rho_1,\quad d_{\mathbb{S}^1}(\varphi_1,\varphi_2)\leqslant d_{\mathbb{S}^1}(\varphi_1,\varphi_3)+d_{\mathbb{S}^1}(\varphi_3,\varphi_2)\leqslant1<2.$$
\qed

\begin{prop} $d_D$ is (topologically) equivalent to the standard Euclidean metric $d_E$.
\end{prop}
\noindent{\it Proof.} We will show that convergent sequences in one metric are also convergent in the other, with the same limit. If the limit point $x$ is distinct from the origin, this is immediate, since the coordinate $\varphi$ is well defined as the continuous function of Euclidean coordinates in a neighbourhood of $x$. Let $x_n=(\rho_n,\varphi_n)$ be a sequence that converges to the origin in the Euclidean metric. Then we have
$$d_D(x_n,0)=\rho_n\to 0,\quad n\to\infty.$$ Conversely, from
$$|\rho_n|\leqslant d_D(x_n,0)$$ we obtain the other implication.\qed

\subsection{Upper and lower bounds for polynomial entropy via ball-dimension}

Let $(X,d)$ be a compact metric space and let $N(A,\delta)$ denote the minimal number of open balls of radius $\delta$ needed to cover a compact set $A\subset X$. The {\it ball-dimension} of the set $A$ is defined by:

$$\displaystyle{\dim}_B(A)=\limsup\limits_{\delta\to0}\frac{\log{N(A,\delta)}}{\log{(1/\delta)}}.$$

Ball-dimension, also known as the Minkowski dimension, is a way of assigning a dimension to spaces that do not have a well-defined dimension in the usual sense. It is a standard result that the ball dimension of an $n$-dimensional topological manifold equals $n$. So, for well-behaved spaces, the ball-dimension agrees with the usual notion of dimension. Let us compute the ball-dimension of the set defined in Example \ref{ex:bd}, since we will need it in the proof of our main Theorem \ref{main}. Similar computations for the set $\{\frac{1}{n^k}\mid n\in\mathbb{N}\}\cup\{0\}$, for some $k\in\mathbb{N}$, have been done. However, we will give detailed computations for the sake of completeness. For $x\in\mathbb{R}$, we denote by $[x]$ the floor function of $x$, which is equal to the greatest integer less than or equal to $x$. Also, let $\{x\}=x-[x]$ be the fractional part of $x$.

\begin{exam}\label{ex:bd} 
Let $a_k=\frac{1}{[k^\alpha]}$, for some $\alpha\in\mathbb{R}$, $\alpha>1$. Then the ball-dimension of the set $A:=\{a_k\mid k\in\mathbb{N}\}\cup\{0\}$ is equal to $\frac{1}{\alpha+1}$. 
\end{exam}
\noindent{\it Proof.} Let us first prove ${\dim}_B(A)\geqslant1/(\alpha+1)$.

Since $(k+1)^\alpha-k^\alpha=\alpha\xi_k^{\alpha-1}$, for some $\xi_k\in(k,k+1)$, and since $\{\cdot\}$ is bounded and $\alpha>1$, we have
$$\begin{aligned}
&\left[(k+1)^\alpha\right]-\left[k^\alpha\right]=(k+1)^\alpha-k^\alpha+\left\{(k+1)^\alpha\right\}-\left\{k^\alpha\right\}
=\\
&\alpha\xi_k^{\alpha-1}+\left\{(k+1)^\alpha\right\}-\left\{k^\alpha\right\}\sim \alpha k^{\alpha-1},
\end{aligned}$$ when $k\to\infty$, where $\sim$ denotes the asymptotic equivalence of two sequences.  
We compute
$$a_k-a_{k+1}=\frac{\left[(k+1)^\alpha\right]-\left[k^\alpha\right]}{\left[(k+1)^\alpha\right]\left[k^\alpha\right]}\sim\frac{\alpha k^{\alpha-1}}{k^{2\alpha}}=\frac{\alpha}{k^{\alpha+1}},$$ when $k\to\infty$.
Choose $k_0\in\mathbb{N}$, such that, for $k\geqslant k_0$
$$
a_k-a_{k+1}\geqslant\frac{1}{2}\frac{\alpha}{k^{\alpha+1}}.
$$
Therefore, for $k\geqslant k_0$ and every $\varepsilon>0$, we have:
$$\frac{\alpha}{k^{\alpha+1}}\geqslant2\varepsilon\quad\Rightarrow\quad a_k-a_{k+1}\geqslant\varepsilon,$$ or, equivalently
$$k\in\left[k_0,(\alpha/2\varepsilon)^{\frac{1}{\alpha+1}}\right]\quad\Rightarrow\quad a_k-a_{k+1}\geqslant\varepsilon.$$
We conclude that at least 
$$\left[(\alpha/2\varepsilon)^{\frac{1}{\alpha+1}}\right]-k_0$$ intervals are needed to cover the subset 
$$\left\{a_k\mid k\in\left[k_0,(\alpha/2\varepsilon)^{\frac{1}{\alpha+1}}\right]\right\}\subset A.$$
Therefore $N(A,\varepsilon)\geqslant\left[(\alpha/2\varepsilon)^{\frac{1}{\alpha+1}}\right]-k_0$ so
$${\dim}_B(A)\geqslant\lim_{\varepsilon\to 0}\frac{\log\left(\left[(\alpha/2\varepsilon)^{\frac{1}{\alpha+1}}\right]-k_0\right)}{\log(1/\varepsilon)}=\frac{1}{\alpha+1}.$$

We now estimate ${\dim}_B(A)$ from above. For a fixed $\varepsilon\in(0,\alpha/2)$, let $k_1:=\left[\left(\frac{\alpha}{2\varepsilon}\right)^\frac{1}{\alpha+1}\right]$. The set $\{a_k\mid k\leqslant k_1\}$ can be covered by $k_1$ intervals. Since $[x]\geqslant x/2$ for $x>1$, we have
$$a_{k_1}=\frac{1}{\left[\left(\frac{\alpha}{2\varepsilon}\right)^\frac{\alpha}{\alpha+1}\right]}\leqslant\frac{1}{\frac12\left(\frac{\alpha}{2\varepsilon}\right)^\frac{\alpha}{\alpha+1}}=2\left(\frac{2\varepsilon}{\alpha}\right)^\frac{\alpha}{\alpha+1}.$$ We can cover the interval 
$$\left[0,2\left(\frac{2\varepsilon}{\alpha}\right)^\frac{\alpha}{\alpha+1}\right]$$ by
$$\frac{2\left(\frac{2\varepsilon}{\alpha}\right)^\frac{\alpha}{\alpha+1}}{\varepsilon}+1=c(\alpha)\varepsilon^{-\frac{1}{\alpha+1}}+1$$ intervals of length $\varepsilon$, where $c(\alpha)$ depends only on $\alpha$ and not on $\varepsilon$. Therefore we conclude
$$N(A,\varepsilon)\leqslant k_1+c(\alpha)\varepsilon^{-\frac{1}{\alpha+1}}+1\leqslant \left(\frac{\alpha}{2\varepsilon}\right)^\frac{1}{\alpha+1}
+c(\alpha)\varepsilon^{-\frac{1}{\alpha+1}}+1=c_1(\alpha)\varepsilon^{-\frac{1}{\alpha+1}}+1.$$ So we obtain
$${\dim}_B(A)=\limsup_{\varepsilon\to 0}\frac{\log N(A,\varepsilon)}{-\log\varepsilon}\leqslant\limsup_{\varepsilon\to 0}\frac{\log(c_1(\alpha)\varepsilon^{-\frac{1}{\alpha+1}}+1)}{-\log\varepsilon}=\frac{1}{\alpha+1}.$$
\qed

\begin{prop}\label{prop:upper}
Let $A\subset[0,1]$ be a compact set such that $0\in A$. We denote by $D_a$ a unit disk in a horizontal plane at height $a$, and with $I$ a vertical line segment which connects the centers of all disks. We define $f:X\to X$ to be a rotation by the angle $a$ on each disk $D_a$, $a\neq0$, $f|_{D_a}:=R_a$, and the identity map on $I$ and $D_{0}$, $f|_{I}:=\mathrm{Id}_I$, $f|_{D_{0}}:=\mathrm{Id}_{D_0}$. Let
$$X:=\bigg(\bigcup_{a\in A}D_a\bigg)\cup I,$$
with metric $d$ defined in the following way:
$$d((x,a_1),(y,a_2))=d_D(x,y)+|a_1-a_2|.$$
Then $h_{\mathrm{pol}}(f,X)\leqslant \dim_{B}(A)$.
\end{prop}
Before proceeding to prove this proposition, let us formulate and prove an auxiliary lemma.

\begin{lem}
Let $D_a$ and $D_b$ be two unit disks on heights $a$ and $b$, respectively, such that $|a-b|<\frac{\varepsilon}{2n}$, for some $n\in\mathbb{N}$ and $\varepsilon<1$. If the set $S=\{(x_1,a),(x_2,a),\ldots,(x_p,a)\}$ is an $\frac{\varepsilon}{2}$-net for the disk $D_a$ in metric $d$, then the same set $S$ is an $\varepsilon$-net for $D_a\cup D_b$ in the dynamic metric $d_n^f$.
\end{lem}

\noindent{\it Proof.}
Since $f$ is a rotation on each disk, and therefore an isometry, we immediately see that $S$ is an $\frac{\varepsilon}{2}$-net for $D_a$ in the dynamical metric $d_n^f$. We have to prove that for any $(y,b)\in D_b$ there exists $(x_i,a)\in S$ such that $d_n^f((y,b),(x_i,a))<\varepsilon$. Denote with $S'=\{(x_1,b),(x_2,b),\ldots,(x_p,b)\}\subset D_b$ the set of points on disk $D_b$ such that points in $S'$ have the same polar coordinates as those in $S$, with the only difference being their height. Then clearly, as for the disk $D_a$, we have that $S'$ is an $\frac{\varepsilon}{2}$-net for $D_b$. Based on this, there exists $(x_i,b)\in S'$ such that $d((y,b),(x_i,b))<\frac{\varepsilon}{2}$, and again, because rotation is an isometry, $d_n^f((y,b),(x_i,b))<\frac{\varepsilon}{2}$. Now using the triangle inequality we obtain:

$$d_n^f((y,b),(x_i,a))\leqslant d_n^f((y,b),(x_i,b))+d_n^f((x_i,b),(x_i,a))<\frac{\varepsilon}{2}+d_n^f((x_i,b),(x_i,a)).$$

The proof is done if we show that $d_n^f((x_i,b),(x_i,a))<\frac{\varepsilon}{2}$. Since $|a-b|<\frac{\varepsilon}{2n}$, we have that $d((x_i,b),(x_i,a))<\frac{\varepsilon}{2n}$. Also, since the rotation doesn't change the distance of $x_i$ from the centre od the disk, we have that only the change of the angle influences the change in the dynamical distance $d_n^f((x_i,b),(x_i,a))$. Then, for all $k\in\mathbb{N}$:
$$d(f^{k}(x_i,a),f^{k}(x_i,b))=d_D(f^{k}(x_i,a),f^{k}(x_i,b))+|a-b|\leqslant0+1\cdot d_{\mathbb{S}^1}(ka,kb)+|a-b|.$$

Since $\varepsilon<1$, i.e., $\varepsilon/2<1/2$, the maximal distance $d_D(f^{k}(x_i,a),f^{k}(x_i,b))$ is achieved for $k=n-1$. In that case, $d_D(f^{k}(x_i,a),f^{k}(x_i,b))\leqslant d_{\mathbb{S}^1}(ka,kb)=d_{\mathbb{S}^1}(0,k|a-b|)\leqslant (n-1)\frac{\varepsilon}{2n}$. Finally, we see that:
$$d_n^f((x_i,a),(x_i,b))=\max\limits_{0\leqslant k\leqslant n-1}d(f^{k}(x_i,a),f^{k}(x_i,b))<(n-1)\frac{\varepsilon}{2n}+\frac{\varepsilon}{2n}=\frac{\varepsilon}{2}$$
\qed\\

\noindent{\it Proof of Proposition \ref{prop:upper}.}  We can construct an $\frac{\varepsilon}{2}-$net $S_a$ in metric $d_D$, for every disk $D_a\subset X$, such that the points in every net have the same first two coordinates, but are at different heights. Denote by $N_D(\frac{\varepsilon}{2})$ the cardinality of such a net. Denote by $N(A,\frac{\varepsilon}{2n})$ the minimal number of open balls $B_j$, $1\leqslant j\leqslant N(A,\frac{\varepsilon}{2n})$, of radius $\frac{\varepsilon}{2n}$ needed to cover $A$. For every $j$, $1\leqslant j\leqslant N(A,\frac{\varepsilon}{2n})$, choose one $a_j\in B_j$ and consider the constructed $\frac{\varepsilon}{2}$-net in metric $d_D$ on the disk $D_{a_j}$. By the previous lemma, for every $j$, the set $S_{a_j}$ is an $\varepsilon-$net for the union of the disks $D_a$, $a\in B_j$ in the dynamic metrics $d_n^f$. We conclude that the union of these nets will be an $\varepsilon-$net of the whole space $X$ with respect to the dynamic metrics $d_n^f$. Consequently, we obtained an $(n,\varepsilon)$-spanning set for $X$ and therefore

$$\mathrm{span}(n,\varepsilon)\leqslant N_D\left(\frac{\varepsilon}{2}\right)N\left(A,\frac{\varepsilon}{2n}\right).$$

It follows that

\begin{align*}
h_{\mathrm{pol}}(f)&\leqslant\lim_{\varepsilon\to 0}\limsup\limits_{n\to+\infty}\frac{\log{\left( N_D\left(\frac{\varepsilon}{2}\right)N\left(A,\frac{\varepsilon}{2n}\right)\right)}}{\log{n}}=\lim_{\varepsilon\to 0}\limsup\limits_{n\to+\infty}\frac{\log{N_D\left(\frac{\varepsilon}{2}\right)}+\log{N\left(A,\frac{\varepsilon}{2n}\right)}}{\log{n}}\\
&=\lim_{\varepsilon\to 0}\limsup\limits_{n\to+\infty}\frac{\log{N\left(A,\frac{\varepsilon}{2n}\right)}}{\log{n}}=\lim_{\varepsilon\to 0}\limsup\limits_{n\to+\infty}\frac{\log{N\left(A,\frac{\varepsilon}{2n}\right)}}{\log{(\frac{2n}{\varepsilon})}\frac{\log{n}}{\log{(\frac{2n}{\varepsilon})}}}\\
&=\lim_{\varepsilon\to 0}\limsup\limits_{n\to+\infty}\frac{N\left(A,\frac{\varepsilon}{2n}\right)}{\log{(\frac{2n}{\varepsilon}})}
\leqslant\lim_{\varepsilon\to 0}\limsup_{\delta\to 0}\frac{\log N(A,\delta)}{\log(1/\delta)}=\dim_B(A).
\end{align*}
\qed

\begin{prop}\label{prop:lower}
Let $(a_k)_{k\in\mathbb{N}}$ be a sequence of distinct elements in $[0,1]$ that converges to zero and $A:=\{a_k\mid k\in\mathbb{N}\}\cup\{0\}$. We denote by $D_a$ a unit disk in a horizontal plane at height $a$, and with $I$ a vertical line segment which connects the centers of all disks. We define $f:X\to X$ to be a rotation by the angle $a$ on each disk $D_a$, $a\neq0$, $f|_{D_a}=R_a$, and the identity map on $I$ and $D_{0}$, $f|_{I}:=\mathrm{Id}_I$, $f|_{D_{0}}:=\mathrm{Id}_{D_0}$. Let
$$X:=\bigg(\bigcup_{a\in A}D_a\bigg)\cup I,$$
with metric $d$ defined in the following way:
$$d((x,a_1),(y,a_2))=d_D(x,y)+|a_1-a_2|.$$

Then $h_{\mathrm{pol}}(f)\geqslant{\mathrm{dim}}_B(A)$. 
\end{prop}
\noindent{\it Proof.} We may assume that $(a_k)$ is strictly decreasing (otherwise, we can relabel the sequence). We also assume that $|a_{j+1}-a_j|<1/2$ (if not, we discard finitely many terms so that this condition holds). 

Fix $\varepsilon_0\in(0,1/2)$. Let $Y:=\{y_1,\ldots,y_r\}\subset D\setminus\{0\}$ be a set of maximal cardinality such that 
$$i\neq j\quad\Rightarrow\quad  d_D(y_i,y_j)\ge\varepsilon_0$$ and, for $a\in A$
$$Y_a:=Y\times\{a\}\subset X.$$ 
Since $f$ acts as a rotation on each level, each $Y_a$ is $(n,\varepsilon_0)$-separated, for all $n\in\mathbb{N}$. Denote by $\rho_0$ the minimal Euclidean norm of the points $y_j$, $j\in\{1,\ldots,r\}$.

Now fix 
\begin{equation}\label{eq:varepsilon}
\varepsilon\leqslant\min\left\{\varepsilon_0,\rho_0/2\right\},
\end{equation} 
$n\in\mathbb{N}$ and choose $\{z_1,\ldots,z_m\}\subset A$ to be a $\varepsilon/(n\rho_0)$-separated set in $A$ of the maximal cardinality.
Define
$$S:=S(n,\varepsilon):=\bigcup_{j=1}^{m} Y_{z_j}.$$

We claim that $S$ is $(n,\varepsilon)$-separated set.
Let $p=(x,a), q=(y,b)\in S$. If $x\neq y$, then
$$d(p,q)\geqslant d_D(x,y)\geqslant\varepsilon_0>\varepsilon.$$ If $x=y=(\rho,\varphi)$, then
$$d(f^j(p),f^j(q))=|a-b|+\rho d_{\mathbb{S}^1}(\varphi+ja,\varphi+jb)\geqslant\rho_0 d_{\mathbb{S}^1}(ja,jb)=\rho_0d_{\mathbb{S}^1}(0,j|a-b|).$$
To conclude that $S$ is $(n,\varepsilon)$-separated we apply the following lemma with $\alpha:=|a-b|$.

\begin{lem} For $\alpha\in [\varepsilon/(n\rho_0),1/2]$ there exists $j\in\{0,\ldots,n-1\}$ such that 
$$d_{\mathbb{S}^1}(0,j\alpha)\geqslant\frac{\varepsilon}{\rho_0}.$$
\end{lem}
\noindent{\it Proof of Lemma.} If $\alpha\geqslant \varepsilon/\rho_0$, we are done. If $\alpha<\varepsilon/\rho_0$, then, since $\varepsilon/\rho_0\leqslant1/2$ (see~(\ref{eq:varepsilon})) we have
$$d_{\mathbb{S}^1}(0,2\alpha)=2\alpha.$$ If $2\alpha\geqslant \varepsilon/\rho_0$ we are done. Otherwise, we repeat the same step. Since $\alpha\geqslant\varepsilon/(n\rho_0)$ the lemma follows.\qed

\vspace{3mm}

We continue the proof of Proposition~\ref{prop:lower}. Since $\{z_1,\ldots,z_m\}$ is the $\varepsilon/(n\rho_0)$-separated set in $A$ of the maximal cardinality, we have
$$|S|=m\cdot r>m\geqslant N(A,\varepsilon/(n\rho_0),$$ where, as before, $N(A,\delta)$ denotes the minimal cardinality of $\delta$-net of $A$.
Since $S$ is $(n,\varepsilon)$-separated, we have
$$\mathrm{sep}(n,\varepsilon)\geqslant |S|\geqslant N(A,\varepsilon/(n\rho_0)),$$ and therefore
$$\begin{aligned}
&\limsup_{n\to\infty}\frac{\log\mathrm{sep(n,\varepsilon)}}{\log n}\geqslant \limsup_{n\to\infty}\frac{\log\ N(A,\varepsilon/(n\rho_0))}{\log n}=
\\
&\limsup_{n\to\infty}\frac{\log N(A,\varepsilon/(n\rho_0))}{\log(n\rho_0/\varepsilon)}\cdot\frac{\log(n\rho_0/\varepsilon)} {\log n}=\\
&\limsup_{n\to\infty}\frac{\log\ N(A,\varepsilon/(n\rho_0))}{\log(n\rho_0/\varepsilon)}=\limsup_{n\to\infty}\frac{N(A,1/n)}{\log n}=\\
&\stackrel{(*)}{=}\limsup_{\delta\to 0}\frac{\log N(A,\delta)}{\log(1/\delta)}={\mathrm{dim}}_B(A).
\end{aligned}$$
Let us explain why the equality $(*)$ holds. A continual upper limit is always greater than or equal to an upper limit of a sequence. We will show the other (nontrivial) inequality. Let $(x_n)_{n\in\mathbb{N}}$, $x_n\to 0$ be a sequence such that
$$\limsup_{\delta\to 0}\frac{\log N(A,\delta)}{\log(1/\delta)}=\lim_{n\to\infty}\frac{\log N(A,x_n)}{\log(1/x_n)}=\mathrm{dim}_B(A).$$
Let $m_n\in\mathbb{N}$ be such that
$$x_n\in\left[\frac{1}{m_n+1},\frac{1}{m_n}\right].$$ Obviousy $m_n\to\infty$, when $n\to\infty$. Since the function $\delta\mapsto N(A,\delta)$ is decreasing and $\log x$ is an increasing function, we have
\begin{equation}\label{eq:m_n}
\frac{\log N(A,1/m_n)}{\log(m_n+1)}\leqslant \frac{\log N(A,x_n)}{\log(1/x_n)}\leqslant \frac{\log N(A,1/(m_n+1))}{\log m_n}.
\end{equation} 
From $m_n\to\infty$ we have
$$\lim_{n\to\infty}\frac{\log(m_n+1)}{\log m_n}=\lim_{n\to\infty}\frac{\log m_n+\log(1+1/m_n)}{\log m_n}=1,$$ and therefore
$$\limsup_{n\to\infty}\frac{\log N(A,1/m_n)}{\log m_n}=\limsup_{n\to\infty}\frac{\log N(A,1/m_n)}{\log(m_n+1)}=\limsup_{n\to\infty}\frac{\log N(A,1/(m_n+1))}{\log m_n}.$$
By taking the upper limit in~(\ref{eq:m_n}), we obtain
$$\limsup_{n\to\infty}\frac{\log N(A,1/m_n)}{\log m_n}=\lim_{n\to\infty}\frac{\log (A,x_n)}{\log(1/x_n)}=\mathrm{dim}_B(A).$$
Now $(*)$ follows from the fact that
$$\limsup_{n\to\infty}\frac{\log N(A,1/n)}{\log n}\geqslant \limsup_{n\to\infty}\frac{\log N(A,1/m_n)}{\log m_n}.$$ \qed

We end this subsection with the following result, which strengthens the Proposition 3.3 obtained in \cite{AOM}.

\begin{prop}\label{prop:ekvi}
There exists a homeomorphism $f:X\to X$ on a continuum $X$ which is not equicontinuous and such that $h_{\mathrm{pol}}(f)=0$.
\end{prop}
\noindent{\it Proof.}
We will apply Proposition \ref{prop:upper} and Proposition \ref{prop:lower}, for the suitably chosen sequence $(a_k)_{k\in\mathbb{N}}$, $a_k=\frac{1}{[e^k]}$ and $A=\{a_k\mid k\in\mathbb{N}\}\cup\{0\}$. Let $f$ and $(X,d)$ be as in Proposition \ref{prop:upper} and let us begin by proving that $f$ is not equicontinuous.

Let $\frac{1}{2}>\varepsilon>0$. For an arbitrary $\delta>0$, choose $k_0$ such that $\frac{1}{[e^{k_0}]}<\delta$ and $[e^{k_0}]$ is an even number. If we take two points $p(1,0,0)$ and $q(1,0,\frac{1}{[e^{k_0}]})$, then $d(p,q)<\delta$. Take $n=\frac{[e^{k_0}]}{2}$. Then
$$f^n(p)=p,\quad f^n(q)=\left(1,\frac{[e^{k_0}]}{2}\frac{1}{[e^{k_0}]},\frac{1}{[e^{k_0}]}\right)=\left(1,\frac{1}{2},\frac{1}{[e^{k_0}]}\right),$$
so we have $d(f^n(p),f^n(q))=\frac{1}{2}+\frac{1}{[e^{k_0}]}>\varepsilon$.

In order to show that $h_{\mathrm{pol}}(f)=0$, it is enough to show that $\dim_B(A)=0$, because of Proposition \ref{prop:upper}. Let $\varepsilon>0$. We are interested for which $k\in\mathbb{N}$ we have that $a_k<\varepsilon$.
$$a_k=\frac{1}{[e^k]}<\varepsilon \iff [e^k]>\frac{1}{\varepsilon}.$$
Since $[x]>\frac{x}{2}$, $[e^k]>\frac{e^k}{2}>\frac{1}{\varepsilon}$ and the last inequality is true for $k>\log{\frac{2}{\varepsilon}}$.

We conclude that for $k>\log{\frac{2}{\varepsilon}}$, elements $a_k$ belong to a ball centered at $0$, with radius $\varepsilon$, so $A$ can be covered with $[\log{\frac{2}{\varepsilon}}]+1$ balls. Now, we can compute the ball-dimension of the set $A$:
$$\dim_B(A)=\limsup_{\varepsilon\to0}\frac{\log N(A,\varepsilon)}{-\log{\varepsilon}}\leqslant\limsup_{\varepsilon\to0}\frac{\log {([\log{\frac{2}{\varepsilon}}]+1)}}{-\log{\varepsilon}}=0.$$

\qed

\subsection{Proof of main result}

We first prove Theorem \ref{main}.

Let $a_k=\frac{1}{[k^\alpha]}$, for some $\alpha\in\mathbb{R}$, $\alpha>1$. and $A=\{a_k\mid k\in\mathbb{N}\}\cup\{0\}$. Define $D_k$ to be the unit disk in a horizontal plane at height $a_k$, $D_0$ is the unit disk in the plane $z=0$ and $I$ is a vertical straight line connecting the centers of all the disks. Then (see the picture below)

$$X:=D_0\cup\bigg(\bigcup\limits_{k\in\mathbb{N}} D_k\bigg)\cup I.$$

\begin{figure}[h]
    \centering 
    \includegraphics[width=0.3\textwidth]{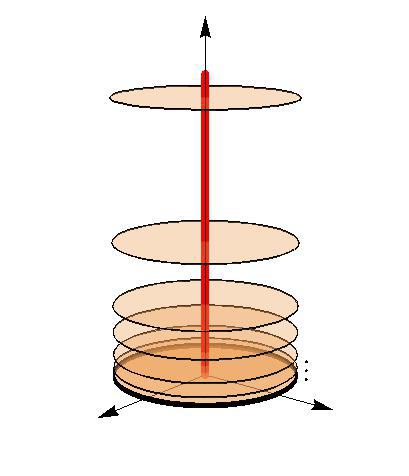}
\end{figure}

Let $R_{a_k}:D_{k}\to D_{k}$ denote the rotation map. Now we define the homeomorphism $f:X\to X$, such that $f|_{D_k}:=R_{a_k}$ and $f|_{I}:=\mathrm{Id}_I$, $f|_{D_{0}}:=\mathrm{Id}_{D_0}$. This is a pointwise periodic map such that $\mathrm{h_{pol}}(f)=\frac{1}{\alpha+1}$, as we will now see. 

We computed the ball-dimension of the set $A=\big\{\frac{1}{[k^{\alpha}]}\mid k\in\mathbb{N}\big\}\cup\{0\}$ in Example \ref{ex:bd} and found it to be $\dim_B(A)=\frac{1}{\alpha+1}$. Now, using Proposition \ref{prop:upper}, we have that
$$h_{\mathrm{pol}}(f,X)\leqslant {\mathrm{dim}}_B(A)=\frac{1}{\alpha+1}.$$
Similarly, using Proposition \ref{prop:lower}, we have that
$$h_{\mathrm{pol}}(f,X)\geqslant {\mathrm{dim}}_B(A)=\frac{1}{\alpha+1}.$$

Let us now proceed to prove Corollary \ref{cor}. Firstly, let $a\neq+\infty$. As we have seen, if $a\in(0,1/2)$, there exists a pointwise periodic homeomorphism $f:X\to X$ on a continuum $X$ such that $h_{\mathrm{pol}}(f)=a$. We obtain it by taking $\alpha=\frac{1}{a}-1$ in Example \ref{ex:bd}. Also, by taking $f$ to be as in proof of Proposition \ref{prop:ekvi}, we get $h_{\mathrm{pol}}(f)=0$. Note that if $f,g:X\to X$ are pointwise periodic homeomorphisms, then $f\times g:X\times X\to X\times X$ is also a pointwise periodic homeomorphism and $h_{\mathrm{pol}}(f\times g)=h_{\mathrm{pol}}(f)+h_{\mathrm{pol}}(g)$. Clearly, we can take a finite product of $l$ maps and get $h_{\mathrm{pol}}(f_1\times f_2\times\ldots\times f_l)=h_{\mathrm{pol}}(f_1)+\ldots+h_{\mathrm{pol}}(f_l)$. In this way we can construct a pointwise periodic homeomorphism with arbitrary finite polynomial entropy.

Lastly, let us construct a continuum $X$ and a pointwise periodic homeomorphism such that $h_{\mathrm{pol}}(f)=+\infty$. Let $(X_m,\rho_m)$ be a compact metric space and $f_m:X_m\to X_m$ a pointwise periodic homeomorphism with at least one fixed point $p_m$ and $h_{pol}(f_m)=m$. We can take, for example, the space constructed in Corollary~\ref{cor}. Now define a new metric $d_m$ on $X_m$ by $d_m(x,y):=1/m\rho_m(x,y)$ and the space $X$ as the wedge sum of the spaces $X_m$. More precisely:
$$X:=\bigsqcup_{m=1}^\infty X_m\;/\sim,\quad\mbox{ where}\quad p_m\sim p_n$$ for all $m,n\in\mathbb{N}$. We now define the metric $d$ on $X$ in the standard way (see, e.g.,~\cite{BH}):
$$d(x,y):=\begin{cases} d_m(x,y), &x,y\in X_m\\
d_m(x,p_m)+d_n(y,p_n),&x\in X_n, y\in X_m.
\end{cases}.$$
It is straightforward to verify that the space $(X,d)$ is a compact metric space and that every $X_m$ can be identified with its closed subset (so we may assume $X_m\subset X$).
Now define $f:X\to X$, as
$$f|_{X_m}:=f_m.$$ It is easy to see that $f$ is a homeomorphism which is pointwise periodic. Since every $X_m$ is $f$-invariant, we have
$$h_{\mathrm{pol}}(f)\geqslant h_{\mathrm{pol}}(f|_{X_m})=m,$$ for every $m$, which then implies $h_{\mathrm{pol}}(f)=\infty$.\qed

\vspace{0.3cm}
We end this paper with an interesting question.\\

\textbf{Question}: Let us remember that pointwise periodic homeomorphisms on local dendrites have zero polynomial entropy. Could there exist a one-dimensional continuum $X$ such that a pointwise periodic homeomorphism $f:X\to X$ has positive polynomial entropy?

\end{document}